\documentclass[11pt, reqno]{amsart}
\usepackage{amsmath, amsthm, amscd, amsfonts, amssymb, graphicx, color}
\usepackage[bookmarksnumbered, colorlinks, plainpages]{hyperref}
\hypersetup{colorlinks=true,linkcolor=red, anchorcolor=green, citecolor=cyan, urlcolor=red, filecolor=magenta, pdftoolbar=true}
\newtheorem{theorem}{Theorem}[section]

\newtheorem{proposition}[theorem]{Proposition}
\newtheorem{corollary}[theorem]{Corollary}
\theoremstyle{definition}
\newtheorem{definition}[theorem]{Definition}
\newtheorem{example}[theorem]{Example}

\theoremstyle{remark}
\newtheorem{remark}[theorem]{Remark}
\numberwithin{equation}{section}

\title[Hypercyclic operators on Hilbert C*-modules]
{Hypercyclic operators on Hilbert C*-modules}

\author[S. Ivkovi\'c]{Stefan Ivkovi\'c}
\address{Mathematical Institute of the Serbian Academy of Sciences and Arts,
	p.p. 367, Kneza Mihaila 36, 11000 Beograd, Serbia}
\email{stefan.iv10@outlook.com}



\date{\today}
\begin{document}

\maketitle

\begin{abstract}
	In this paper we characterize hypercyclic generalized bilateral weighted shift operators on the standard Hilbert module over the $C^{*}$-algebra of compact operators on the separable Hilbert space. Moreover, we give necessary and sufficient conditions for these operators to be chaotic and we provide concrete examples.
\end{abstract}

\baselineskip17pt

\section{Introduction}
 In last decades, linear dynamical properties of operators  have been studied in many research articles; see monographs \cite{bmbook,gpbook}, and recent papers \cite{IT,tsi}. Among other concepts, hypercyclicity, topological transitivity and topological mixing, as important linear dynamical properties of bounded linear operators, have been investigated in many research works; see \cite{an97, bg99, kuchen17, shaw, kostic} and their references.
 Specially, hypercyclic weighted shifts on $\ell^p(\Bbb{Z})$ were characterized in \cite{sa95,ge00}, and then  C-C. Chen and 
 C-H. Chu, using aperiodic elements of locally compact groups, extended the results in \cite{sa95} to weighted translations on Lebesgue spaces in the context of a second countable group \cite{CC}. Afterwards, many other linear dynamics in connection with this theme have been studied; see \cite{chen11, chen141, chta1, chta2, chta3}.

Recently, in \cite{si09} we have for instance characterized hypercyclic weighted composition operators on the commutative $C^{*}$-algebra of continuous functions vanishing at infinity on a locally compact, non-compact Hausdorff space. Moreover, in \cite{IT} we have characterized hypercyclic elementary operators on the space of compact linear operators on a separable Hilbert space. The dynamics of some similar operators have been considered earlier such as conjugate operators, see \cite{pet}, and left multiplication operators, see \cite{ZZD,ZLD,YRH}.

In this paper, in Section \ref{Sec 4} we provide an algebraic generalization of our previous results given in \cite{si09,IT} to arbitrary non-unital $C^{*}$-algebras.

However, the main aim of this paper is to study the dynamics of  generalized bilateral weighted shift operators on the standard Hilbert $C^{*}$-module over the $C^{*}$-algebra of compact operators on a separable  Hilbert space, thus to generalize in this setting the results from \cite{sa95,ge00}. In Section \ref{Section 3} we characterize hypercyclic such operators and we also give necessary and sufficient conditions for these operators to be chaotic. Moreover, we provide concrete examples.

\section{Preliminaries}

If $\mathcal X$ is a Banach space, the set of all bounded linear operators from $\mathcal X$ into $\mathcal X$ is denoted by $B(\mathcal X)$. Also, we denote $\mathbb{N}_0:=\mathbb{N}\cup\{0\}$.

\begin{definition}\cite[Definition 2.1]{tsi}
	Let $\mathcal X$ be a Banach space. A sequence $(T_n)_{n\in \mathbb{N}_0}$  of operators in $B(\mathcal X)$ is called {\it topologically transitive} if for each non-empty open subsets $U,V$ of
	${\mathcal X}$, $T_n(U)\cap V\neq \varnothing$ for some $n\in \mathbb{N}$. If $T_n(U)\cap V\neq \varnothing$ holds from some $n$ onwards, then
	$(T_n)_{n\in \mathbb{N}_0}$ is called {\it topologically mixing}. 
\end{definition}
\begin{definition}\cite[Definition 2.2]{tsi}
	Let $\mathcal X$ be a Banach space. A sequence $(T_n)_{n\in \mathbb{N}_0}$  of operators in $B(\mathcal X)$ is called {\it hypercyclic} if there is an element $x\in\mathcal X$ (called \emph{hypercyclic vector}) such that the orbit $\{T_nx:\,n\in\mathbb N_0\}$ is dense in $\mathcal X$. The set of all hypercyclic vectors of a sequence $(T_n)_{n\in \mathbb{N}_0}$ is denoted by $HC((T_n)_{n\in \mathbb{N}_0})$. If $HC((T_n)_{n\in \mathbb{N}_0})$ is dense in $\mathcal X$, the sequence $(T_n)_{n\in \mathbb{N}_0}$ is called \emph{densely hypercyclic}. An operator $T\in B(\mathcal X)$ is called \emph{hypercyclic} if the sequence $(T^n)_{n\in \mathbb N_0}$ is hypercyclic.
\end{definition}

Note that a sequence $(T_n)_{n\in \Bbb{N}_0}$  of operators in $B(\mathcal X)$ is topologically transitive if and only if it is densely hypercyclic \cite{gro}. Also, a Banach space admits a hypercyclic operator if and only if it is separable and infinite-dimensional \cite{an97,bg99}.

\begin{definition} \cite[Definition 2.3]{tsi}
	Let $\mathcal X$ be a Banach space, and $(T_n)_{n\in \Bbb{N}_0}$  be a sequence of operators in $B(\mathcal X)$. A vector $x\in \mathcal X$ is called a {\it periodic element} of $(T_n)_{n\in \Bbb{N}_0}$ if there exists a constant $N\in\mathbb N$ such that for each $k\in\mathbb N$, $T_{kN}x=x$. The set of all periodic elements of $(T_n)_{n\in \Bbb{N}_0}$ is denoted by
	${\mathcal P}((T_n)_{n\in \Bbb{N}_0})$. The sequence $(T_n)_{n\in \Bbb{N}_0}$ is called {\it chaotic} if $(T_n)_{n\in \Bbb{N}_0}$ is topologically transitive and ${\mathcal P}((T_n)_{n\in \Bbb{N}_0})$ is dense in ${\mathcal X}$. An operator $T\in B(\mathcal{X})$ is called \emph{chaotic} if the sequence $\{T^n\}_{n\in \Bbb{N}_0}$ is chaotic. 
\end{definition}

\section{Generalized weighted bilateral shift operators over $C^{*}$-algebras}\label{Section 3}

Let $H$ be a separable Hilbert space. The $C^{*}$-algebra of all bounded linear operators on $ H$ is denoted by $B( H)$ whereas we let $\mathcal{A}:=B_0( H)$ be the $C^{*}$-algebra of all compact operators on $ H$. For every $T,S\in B( H)$ we denote $T\leq S$ whenever $\langle (T-S)h,h\rangle\geq 0$ for all $h\in H$. Assume that $\{e_j\}_{j\in\mathbb{Z}}$ is an orthonormal basis for $ H$, and for each $m\in\mathbb{N}$, $P_m$ is the orthogonal projection onto 
${\rm Span}\{e_{-m},\ldots,e_{m}\}$.
 Let $W:=\{W_j\}_{j\in\mathbb{Z}}$ be a uniformly bounded sequence of invertible operators in $B( H)$ such that the sequence $\{W_j^{-1}\}_{j\in\mathbb{Z}}$ is also uniformly bounded in $B( H)$. Moreover, let $U$ be a unitary operator on $ H$. We define $T_{U,W}$ to be the operator on $\ell_2(\mathcal{A})$, the standard right Hilbert module over $\mathcal A$, given by 
$$(T_{U,W}(x))_\xi:=W_\xi\,x_{\xi-1}\,U$$
for all $\xi\in\mathbb{Z}$ and $x:=(x_j)_{j\in\mathbb{Z}}\in \ell_2(\mathcal{A})$. It is easy to see that $T_{U,W}$ is a linear operator. Put $M:=\sup_{j\in\mathbb Z}\|W_j\|$. Then, since for all $j\in\mathbb{Z}$,\, 
$M^2U^*x_{j-1}^*x_{j-1}U-U^*x_{j-1}^*W_j^*W_jx_{j-1}U$ is a positive semidefinite operator on $ H,$ we have 
\begin{align*}
\sum_{j\in\mathbb{Z}}U^*x_{j-1}^*W_j^*W_jx_{j-1}U&\leq M^2 \sum_{j\in\mathbb{Z}}U^*x_{j-1}^*x_{j-1}U\\
&=M^2\,U^*(\sum_{j\in\mathbb{Z}}x_{j-1}^*x_{j-1})U\\
&=M^2\,U^*\,\langle x,x\rangle\, U,
\end{align*}
so ${\rm Im} T_{U,W}\subseteq \ell_2(\mathcal{A})$ and $\|T_{U,W}\|\leq M$. Moreover, $T_{U,W}$ is invertible and its inverse $S_{U,W}$ is given by 
$$(S_{U,W}(y))_\xi:=W_{\xi+1}^{-1}y_{\xi+1}U^*$$
for all $y:=(y_j)_j\in\ell_2(\mathcal{A})$ and $\xi\in\mathbb{Z}$. By some calculations we can see that $$(T_{U,W}^n(x))_\xi=W_\xi W_{\xi-1}\ldots W_{\xi-n+1}x_{\xi-n} U^n$$
and 
$$(S_{U,W}^n(y))_\xi:=W_{\xi+1}^{-1}W_{\xi+2}^{-1}\ldots W_{\xi+n}^{-1}y_{\xi+n}U^{*n}$$
for all $n\in\mathbb{N}$, $\xi\in\mathbb{Z}$ and $x:=(x_j)_j,y:=(y_j)_j$ in $\ell_2(\mathcal{A})$. 

For each $J\in\mathbb N$, we denote $[J]:=\{-J,-J+1,\ldots,J-1,J\}$. In the following result, we give some equivalent condition for a sequence of powers of an operator $T_{U,W}$ to be densely hypercyclic on $\ell_2(\mathcal{A})$.
\begin{proposition} \label{Rni3 p2.1}
	Let $(t_n)_n$ be an unbounded sequence of nonnegative integers. We denote $T_{U,W,n}:=T_{U,W}^{t_n}$ for all $n\in\mathbb{N}$. Then, the followings are equivalent:
	\begin{enumerate}
		\item $(T_{U,W,n})_n$ is a densely hypercyclic sequence on $\ell_2(\mathcal{A})$.
		\item For every $J,m\in\mathbb{N}$ there exist a strictly increasing sequence $\{n_k\}_k\subseteq\mathbb{N}$ and  sequences $\{D_{i}^{(k)}\}_k$ and $\{G_i^{(k)}\}_k$ for all $i\in [J]$ of operators in $B_0( H)$ such that 
	$$\lim_{k \rightarrow \infty}\|D_j^{(k)}-P_m\|=\lim_{k \rightarrow \infty} \|G_j^{(k)}-P_m\|=0$$
	and
	$$\lim_{k \rightarrow \infty}\|W_{j+t_{n_k}}W_{j+t_{n_k}-1}\ldots W_{j+1}D_j^{(k)}\|$$
	$$= \lim_{k \rightarrow \infty}\|W_{j-t_{n_k}+1}^{-1}W_{j-t_{n_k}+2}^{-1}\ldots W_j^{-1}G_j^{(k)}\|=0$$
	for all $j\in[J]$.
	\end{enumerate}
\end{proposition}
\begin{proof}
$(1)\Rightarrow(2)$:  Let $(T_{U,W,n})_n$ be densely hypercyclic. Assume that $J,m\in\mathbb{N}$, and define $x=(x_j)_j\in \ell_2(\mathcal{A})$ by $x_j:=P_m$ for all $j\in [J]$, and $x_j:=0$ for all $j\notin[J]$. Then, for each $k\in\mathbb{N}$, there exist an element $y^{(k)}\in \ell_2(\mathcal{A})$ and a term $t_{n_k}$ such that $\|y^{(k)}-x\|_2<\frac{1}{4^k}$ and 
	$\|T_{U,W,n_k}(y^{(k)})-x\|_2<\frac{1}{4^k}$. We can assume that the sequence $(n_k)_k$ is strictly increasing, and  $2J<t_{n_1}<t_{n_2}<\ldots$. Hence, 
	$$\|W_jW_{j-1}\ldots W_{j-t_{n_k}+1} y_{j-t_{n_k}}U^{t_{n_k}}-P_m\|<\frac{1}{4^k}$$
	for all $j\in[J]$. However, since $t_{n_k}>2J$ and $\|y^{(k)}-x\|_2<\frac{1}{4^k}$, we have $\|y^{(k)}_{j-t_{n_k}}\|<\frac{1}{4^k}$ as $x_{j-t_{n_k}}=0$ for all $j\in[J]$. Thus 
	$$\|W_{j-t_{n_k}+1}^{-1}\ldots W_j^{-1}W_j\ldots W_{j-t_{n_k}+1} y^{(k)}_{j-t_{n_k}}U^{t_{n_k}}\|=\|y^{(k)}_{j-t_{n_k}}U^{t_{n_k}}\|=\|y^{(k)}_{j-t_{n_k}}\|<\frac{1}{4^k}$$
	for all $j\in[J]$. Similarly, since $\|T_{U,W}^{t_{n_k}}(y^{(k)})-x\|_2<\frac{1}{4^k}$, we have 
	$$\|W_{j+t_{n_k}}\ldots W_{j+1}y_{j}^{(k)}U^{t_{n_k}}\|<\frac{1}{4^k},$$
	so $\|W_{j+t_{n_k}}\ldots W_{j+1}y_{j}^{(k)}\|<\frac{1}{4^k}$. Set $$D_j^{(k)}:=y_j^{(k)}\qquad\text{and}\qquad G_j^{(k)}:=W_jW_{j-1}\ldots W_{j-t_{n_k}+1}y_{j-t_{n_k}}^{(k)}U^{t_{n_k}}$$
	 for all $j\in [J]$. Then, $$\|D_j^{(k)}-P_m\|<\frac{1}{4^k}, \text{ } \|G_j^{(k)}-P_m\|<\frac{1}{4^k},\text{ }\|W_{j+t_{n_k}}\ldots W_{j+1}D_j^{(k)}\|<\frac{1}{4^k}$$ and $\|W_{j-t_{n_k}+1}^{-1}\ldots W_j^{-1}G_j^{(k)}\|<\frac{1}{4^k}$. Notice that since the coefficients of $y^{(k)}$ belong to $\mathcal{A}=B_0( H)$ which is an ideal of $B( H)$, by construction, $D_j^{(k)}$ and $G_j^{(k)}$ belong to $B_0( H)$ for all $j\in[J]$. This completes the proof. 
	 
	 $(2)\Rightarrow(1)$: Assume that the condition (2) holds. Choose two non-empty open subsets $\mathcal{O}_1$ and $\mathcal{O}_2$ of $\ell_2(\mathcal{A})$. Assume that
	 $\mathcal{F}$ denotes the set of all elements $x=(x_j)_j\in\ell_2(\mathcal{A})$ such that for some $J,m\in\mathbb{N}$, $x_j=0$ for all $j\notin[J]$ and $x_j=P_mx_j$ for all $j\in[J]$. Since $\mathcal F$ is dense in $\ell_2(\mathcal{A})$ \cite[Proposition 2.2.1]{MT}, we can find some $x=(x_j)_j\in \mathcal{O}_1$ and $y=(y_j)_j\in \mathcal{O}_2$ and sufficiently large $J,m$ such that $x_j=y_j=0$ for all $j\notin[J]$ and $x_j=P_mx_j$ and $y_j=P_my_j$ for all $j\in[J]$. Choose the sequences $\{D_j^{(k)}\}_k$ and $\{G_j^{(k)}\}_k$  for $j\in [J]$ and the increasing sequence $\{n_k\}_k$ satisfying (ii) regarding these $J,m$. For each $k$, let $u_k$ and $v_k$  be sequences in $\ell_2(\mathcal{A})$ defined by  $(u_k)_j:=D_j^{(k)}x_j$ for $j\in[J]$, $(u_k)_j:=0$ for $j\notin [J]$,  $(v_k)_j:=G_j^{(k)}y_j$ for $j\in[J]$ and $(v_k)_j:=0$ for all $j\notin[J]$. Set 
	 $$\eta_k:=u_k+S_{U,W}^{t_{n_k}}v_k.$$
	 Since $\|D_j^{(k)}-P_m\|\rightarrow 0$ and $\|G_j^{(k)}-P_m\|\rightarrow 0$ as $k$ tends to $\infty$, and $x_j=P_mx_j$ and $y_j=P_my_j$ for $j\in[J]$, it would be routine to see that $u_k\rightarrow x$ and $v_k\rightarrow y$ as $k\rightarrow \infty$. Next, for each $j\in [J]$ we have 
	\begin{align*}
	\|(S^{t_{n_k}}_{U,W}(v_k))_{j-t_{n_k}}\|&=\|W_{j+1-t_{n_k}}^{-1}\ldots W_j^{-1}G_j^{(k)}y_jU^{-t_{n_k}}\|\\
	&\leq \|W_{j+1-t_{n_k}}^{-1}\ldots W_j^{-1}G_j^{(k)}\|\,\|y_j\|\to 0,	
	\end{align*}
	as $k\to \infty$. On the other hand, for each $j\notin [J]$ we have $(S^{{t_{n_k}}}_{U,W}(v_k))_{j-{{t_{n_k}}}}=0$. Thus, $S^{t_{n_k}}_{U,W}(v_k)\to 0$ as $k\to \infty$. Similarly, since
	$$\parallel T_{U,W}^{t_{n_{k}}} (\mu_{k})_{j+t_{n_{k}}} \parallel = \parallel W_{j+t_{n_{k}}} \dots W_{j}D_{j}^{(k)}x_{j}U^{t_{n_{k}}}  \parallel   $$ 
	$$\leq \parallel  W_{j+t_{n_{k}}} \dots W_{j}D_{j}^{(k)} \parallel \text{ } \parallel x_{j} \parallel \rightarrow 0 $$  
	as $k \rightarrow \infty  $ for all $j \in [J]$ and $T_{U,W}^{t_{n_{k}}} (\mu_{k})_{j+t_{n_{k}}} =0$
	for $j \notin [J]$, we have that $T^{t_{n_k}}_{U,W}(u_k)\to 0$ as $k\to\infty$. It follows that $\eta_k\to x$ and $T^{t_{n_k}}_{U,W}(\eta_k)\to y$ as $k\to\infty$. Hence, the sequence $(T_{U,W,n})_n$ is topologically transitive, and thus it is densely hypercyclic on $\ell_2(\mathcal{A})$.
\end{proof}
\begin{theorem}\label{thm2}
		Let $(t_n)_n$ be an unbounded sequence of nonnegative integers. 
	Suppose that for every $j\in\mathbb{Z}$ there exist subsets $ H_j^{(1)}$ and $ H_j^{(2)}$ of $ H$ and a strictly increasing sequence $(n_k)_k\subseteq\mathbb{N}$ such that 
	$$\lim_{k \rightarrow \infty}W_{j+t_{n_k}}W_{j+t_{n_k}-1}\ldots W_{j+1}=0\quad\text{pointwise on } H_j^{(1)}$$
	and 
	$$\lim_{k \rightarrow \infty}W^{-1}_{j-t_{n_k}+1}W^{-1}_{j-t_{n_k}+2}\ldots W^{-1}_{j}=0\quad\text{pointwise on } H_j^{(2)}$$
	for all $j\in\mathbb{Z}$. Then, the sequence $(T_{U,W,n})_n$ is densely hypercyclic on $\ell_2(\mathcal{A})$, where  $T_{U,W,n}:=T_{U,W}^{t_n}$ for all $n\in\mathbb{N}$.
\end{theorem}
\begin{proof}
	Assume that $m,J\in\mathbb{N}$. Since for each $j\in\mathbb{Z}$, $ H_j^{(1)}$ and $ H_j^{(2)}$ are dense in $ H$, we can find sequences $(f_{i,l}^{(j)})_i\subseteq  H_j^{(1)}$ and $(g_{i,l}^{(j)})_i\subseteq  H_j^{(2)}$ such that $f_{i,l}^{(j)}\to e_l$ and $g_{i,l}^{(j)}\to e_l$ as $i\to\infty$ for all $j\in[J]$ and $l\in[m]$. By the assumptions, one can construct a subsequence $(n_{k_i})_i$ such that 
	$$\|W_{j+t_{n_{k_i}}}W_{j+t_{n_{k_i}}-1}\ldots W_{j+1}f_{i,l}^{(j)}\|<\frac{1}{2mi}$$
	and 
	$$\|W^{-1}_{j-t_{n_{k_i}}+1}W^{-1}_{j-t_{n_{k_i}}+2}\ldots W_{j}^{-1}g_{i,l}^{(j)}\|<\frac{1}{2mi}$$
	for all $j\in[J]$ and $l\in[m]$. For each $j\in[J]$ define the operators $D_j^{(i)}$ and $G_j^{(i)}$ as 
	$$D_j^{(i)}e_l:=\left\{
	\begin{array}{ll}
	f_{i,l}^{(j)},& \mbox{ if } l\in[m]\\\\
	0, & \mbox{ if } l\notin[m]
	\end{array}
	\right.
	\quad\text{and}\qquad
	G_j^{(i)}e_l:=\left\{
	\begin{array}{ll}
	g_{i,l}^{(j)},& \mbox{ if } l\in[m]\\\\
	0, & \mbox{ if } l\notin[m].
	\end{array}
	\right.$$
	By using the fact that strong convergence and uniform convergence coincide on finite dimensional subspaces, we can do same as in the proof of \cite[Theorem 2.10]{IT2}.
\end{proof}
\begin{example} \label{Rni3 e2.3}
	Let $ H:=L^2(\mathbb{R})$. Assume that $(w_j)_{j\in\mathbb Z}\subseteq L^\infty(\mathbb{R})$ such that each $w_j$ is positive and invertible in $L^\infty(\mathbb R)$, and also there exists an $M>0$ such that $\|w_j\|_\infty,\|w_j^{-1}\|_\infty\leq M$ for all $j\in\mathbb Z$. Assume in addition that there exists an $\epsilon>0$ such that $|w_j\chi_{[0,\infty)}|\leq 1-\epsilon$ for all $j\geq 0$ and $|w_j\chi_{(-\infty,0)}|\geq 1+\epsilon$ for all $j<0$. Let $(r_j)_j$ be a sequence of positive numbers such that $r_j\geq C$ for all $j\in\mathbb{Z}$ and some $C>0$. For each $j\in\mathbb{Z}$ let $\alpha_j$ to be the translation on $\mathbb R$ given by $\alpha_j(t):=t-r_j$. For each $j\in\mathbb Z$ assume that $W_j$ is an operator on $L^2(\mathbb R)$ defined by 
	$$W_j(f):=w_j\,(f\circ\alpha_j)$$ for every $f\in L^2(\mathbb R)= H$. Then, each $W_j$ is invertible in $B( H)$, and $\|W_j\|,\,$ $\|W_j^{-1}\|\leq M$. By some calculations we have 
	\begin{align*}
	W_{j+n}W_{j+n-1}\ldots W_jf&=w_{j+n}(w_{j+n-1}\circ\alpha_{j+n})\ldots\\
	&(w_{j}\circ\alpha_{j+1}\circ\ldots\alpha_{j+n})(f\circ\alpha_j\circ\ldots\circ\alpha_{j+n})
	\end{align*}
	for all $f\in  H$ and $j,n\in\mathbb N$. It follows that 
	\begin{align*}
	\|&W_{j+n}W_{j+n-1}\ldots W_jf\|\leq \\
	&\sup_{t\in{\rm supp}f}\big((w_{j+n}\circ \alpha_{j+n}^{-1}\circ\ldots\circ \alpha_j^{-1})(w_{j+n-1}\circ\alpha_{j+n-2}^{-1}\circ\ldots\circ\alpha_j^{-1})\ldots\\
	&(w_j\circ\alpha_j^{-1})\big)(t)\,\|f\|
	\end{align*}
	for all $f\in  H$ and $j,n\in\mathbb{N}$. Similarly, since for each $j$ we have $W_j^{-1}(f)=(w_j^{-1}\circ\alpha_j^{-1})\,(f\circ\alpha_j^{-1})$ for all $f\in  H$, we get that 
	\begin{align*}
	W_{j-n+1}^{-1}W_{j-n+2}^{-1}\ldots W_j^{-1}f&=(w_{j-n+1}^{-1}\circ\alpha_{j-n+1}^{-1})\,(w_{j-n+2}^{-1}\circ\alpha_{j-n+2}^{-1}\circ\alpha_{j-n+1}^{-1})\ldots\\ &\qquad(w_j^{-1}\circ\alpha_j^{-1}\circ\ldots\circ\alpha_{j-n+1}^{-1})(f\circ\alpha_j^{-1}\circ\ldots\circ\alpha_{j-n+1}^{-1})
	\end{align*}
	for all $f\in  H$ and $j,n\in\mathbb N$. Hence,
	$$\|W_{j-n+1}^{-1}W_{j-n+2}^{-1}\ldots W_j^{-1}f\| $$
	$$\leq \sup_{t\in{\rm supp}f}\Big((w_{j-n+1}^{-1}\circ\alpha_{j-n+2}\circ\ldots\circ\alpha_j)\,(w_{j-n+2}^{-1}\circ\alpha_{j-n+3}\circ\ldots\circ\alpha_{j})\ldots w_j\Big)(t)\,\|f\| $$

	for all $f\in  H$ and $j,n\in\mathbb N$. It follows that for every $j\in\mathbb Z$, the sequences $(W_{j+n}\dots W_j)_n$ and $(W_{j-n+1}\ldots W_j^{-1})_n$ converge pointwise on $C_c(\mathbb{R})$ which is dense in $L^2(\mathbb{R})$. Hence, the conditions  in Theorem \ref{thm2} are satisfied.

	In fact, it sufficies to assume that there exist two strictly increasing sequences  $\lbrace n_{k} \rbrace_{k}, \lbrace n_{i} \rbrace_{i} \in \mathbb{N} $ such that for each  $j \in \lbrace n_{k} \rbrace_{k} \cup \lbrace -n_{i} \rbrace_{i} $ the operator  $W_{j} $ is constructed as above. If, for all  $j \in \mathbb{Z} \setminus ( \lbrace n_{k} \rbrace_{k} \cup \lbrace -n_{i} \rbrace_{i}) ,$ we have that  $W_{j} (f) = w_{j}f $  for all  $ f \in H$ where  $w_{j} $  is a function on  $\mathbb{R} $  satisfying that  $\dfrac{1}{M} \leq \vert w_{j} \vert \leq 1 $  whenever  $j \geq 0 $  and  $M \geq \vert w_{j} \vert \geq 1 $  whenever  $j <0 ,$ then it is not hard to see that the conditions of Theorem \ref{thm2} are still satisfied. 		
		
\end{example}

\begin{proposition} \label{Rni3 p3.6}
	We have $(ii) \Rightarrow (i) .$\\
	(i) The operator $T_{U,W} $ is chaotic.\\ 
	(ii) For every $J,m \in \mathbb{N} $ there exist a strictly increasing sequence $\lbrace n_{k} \rbrace_{k} \subseteq \mathbb{N} $ and a sequence $\lbrace D_{i}^{(k)} \rbrace_{k}$ for $i \in [J] $ of operators in $B_{0}(H) $ such that 
	$$\lim_{k \rightarrow \infty} \parallel D_{i}^{(k)} - P_{m} \parallel = 0  $$ 
	and 
	$$\lim_{k \rightarrow \infty} \sum_{l=1}^{\infty} \parallel W_{j+ln_{k}} W_{j+ln_{k}-1} \dots W_{j+1}D_{j}^{(k)} \parallel $$ 
	$$= \lim_{k \rightarrow \infty} \sum_{l=1}^{\infty} \parallel W_{j-ln_{k}+1}^{-1} W_{j-ln_{k}+2}^{-1} \dots W_{j}^{-1}D_{j}^{(k)} \parallel = 0,$$ 
	for all $j \in [J],$  where the corresponding series are convergent for each $k.$
\end{proposition}

\begin{proof}
	By Proposition \ref{Rni3 p2.1}. it suffices to show that $\mathcal{P}(T_{U,W}^{n})_{n} $ is dense in $\ell_2(\mathcal{A}).$ Let $\mathcal{O} $ be an open subset of $\ell_2(\mathcal{A}) $ and $ x=(x_{j})_{j \in \mathbb{Z}} \in \mathcal{O}.$ Then there exist some $J,m \in \mathbb{N}$ such that $y \in \mathcal{O} $ with
	$$y_{j}=\left\{
	\begin{array}{ll}
	P_{m}x_{j},& \mbox{ for } j\in[J],\\\\
	0, & \mbox{ else } .
	\end{array}
	\right.
	$$
	Choose sequences $\lbrace n_{k} \rbrace_{k} $ and $\lbrace D_{i}^{(k)} \rbrace_{k}$ with $i \in [J] $  that satisfy the assumptions in (ii) with respect to $J $ and $m.$ For each $k \in \mathbb{N} ,$ set $Z^{(k)} = (Z_{j}^{(k)})_{j \in \mathbb{Z}}$ to be given by 
	$$Z_{j}^{(k)}=\left\{
	\begin{array}{ll}
	D_{j}^{(k)}y_{j},& \mbox{ for } j\in[J],\\\\
	0, & \mbox{ else } ,
	\end{array}
	\right.
	$$
	and put 
	$$ q_{k}= \sum_{l=0}^{\infty} T_{U,W}^{ln_{k}}(Z^{(k)}) + \sum_{l=1}^{\infty} S_{U,W}^{ln_{k}}(Z^{(k)}).$$ 
	Now, as in the proof of Proposition \ref{Rni3 p2.1} part (2) implies (1), we observe that for each $j \in [J] $ and $l,k \in \mathbb{N} $ we have 	
	$$\parallel  T_{U,W}^{ln_{k}}(Z^{(k)})_{j-ln_{k}}  \parallel \leq
	\parallel  W_{j+ln_{k}} W_{j+ln_{k}-1}  \dots W_{j+1}D_{j}^{(k)} \parallel \parallel y_{j} \parallel ,$$ 
	and 
	$$\parallel  S_{U,W}^{ln_{k}}(Z^{(k)})_{j-ln_{k}}  \parallel \leq
	\parallel  W_{j-ln_{k}+1}^{-1} W_{j-ln_{k}+2}^{-1}  \dots W_{j}^{-1}D_{j}^{(k)} \parallel \parallel y_{j} \parallel ,$$ 
	whereas for $j \notin [J] $ we have that 
	$$T_{U,W}^{ln_{k}}(Z^{(k)})_{j-ln_{k}} =  S_{U,W}^{ln_{k}}(Z^{(k)})_{j-ln_{k}} = 0 .$$ 
	So 
	$$\parallel q_{k} - y \parallel \leq \parallel D_{(0)}^{(k)} - P_{m} \parallel \parallel y_{0} \parallel $$ $$+ \sum_{l=1}^{\infty} \sum_{j \in [J]} \parallel  W_{j+ln_{k}} W_{j+ln_{k}-1}  \dots W_{j+1}D_{j}^{(k)} \parallel \parallel y_{j} \parallel $$
	
	$$+ \sum_{l=1}^{\infty} \sum_{j \in [J]} \parallel  W_{j-ln_{k}+1}^{-1} W_{j-ln_{k}+2}^{-1}  \dots W_{j}^{-1}D_{j}^{(k)} \parallel \parallel y_{j} \parallel $$
	
	$$\leq  \parallel D_{(0)}^{(k)} - P_{m} \parallel \parallel y_{0} \parallel $$ $$+ \sum_{j \in [J]}     \parallel y \parallel (\sum_{l=1}^{\infty} \parallel  W_{j+ln_{k}} W_{j+ln_{k}-1}  \dots W_{j+1}D_{j}^{(k)} \parallel  $$
	
	$$+ \sum_{l=1}^{\infty} \parallel W_{j-ln_{k}+1}^{-1} W_{j-ln_{k}+2}^{-1}  \dots W_{j}^{-1}D_{j}^{(k)} \parallel  )  $$ 
	for all $k \in \mathbb{N} ,$ which gives that $q_{k} \rightarrow y $ as $k \rightarrow \infty .$\\
	Moreover, it is straightforward to check that $T_{U,W}^{ln_{k}} (q_{k}) = q_{k} $ for all $l$ and $k,$ hence $q_{k} \in \mathcal{P}(T_{U,W}^{n})_{n} $ for all $k.$
\end{proof}

Next, for each  $ n \in \mathbb{N},$ we set $C_{U,W}^{(n)} = \dfrac{1}{2} (T_{U,W}^n +S_{U,W}^n ). $ 

\begin{proposition} \label{ni pr3.5}
	We have that $(ii)$ implies $(i)$.\\
	(i) The sequence  $ \lbrace C_{U,W}^{(n)} \rbrace_{n} $ is topologically transitive on  $l_{2}(\mathcal{A})  .$\\
	(ii) For every  $J,m \in \mathbb{N} $  there exist a strictly increasing sequence 
	$ \lbrace n_{k} \rbrace_{k} \subseteq \mathbb{N} $ and sequences of operators  $\lbrace D_{i}^{(k)} \rbrace_{k} ,$ $\lbrace G_{i}^{(k)} \rbrace_{k} $ in  $B_{0}(H) $  for  $i \in [J] $  such that for all  $j \in [J] $ we have that
	
	$$\lim_{k \rightarrow \infty} \parallel D_{j}^{(k)} - P_{m} \parallel =\lim_{k \rightarrow \infty} \parallel G_{j}^{(k)} - P_{m} \parallel=0, $$
	$$\lim_{k \rightarrow \infty} \parallel W_{j+n_{k}} W_{j+n_{k}-1} \dots W_{j+1} D_{j}^{(k)} \parallel $$ 
	$$=  
	\lim_{k \rightarrow \infty}\parallel  W_{j-n_{k}+1}^{-1} W_{j-n_{k}+2}^{-1} \dots W_{j}^{-1} D_{j}^{(k)} \parallel    $$
	
	$$=\lim_{k \rightarrow \infty} \parallel W_{j+n_{k}} W_{j+n_{k}-1} \dots W_{j+1} G_{j}^{(k)} \parallel $$ 
	$$=  
	\lim_{k \rightarrow \infty}\parallel  W_{j-n_{k}+1}^{-1} W_{j-n_{k}+2}^{-1} \dots W_{j}^{-1} G_{j}^{(k)} \parallel  = 0 $$
	and
	$$\lim_{k \rightarrow \infty} \parallel W_{j+2n_{k}} W_{j+2n_{k}-1} \dots W_{j+1} G_{j}^{(k)} \parallel $$ 
	$$= \lim_{k \rightarrow \infty}\parallel  W_{j-2n_{k}+1}^{-1} W_{j-2n_{k}+2}^{-1} \dots W_{j}^{-1} G_{j}^{(k)} \parallel =0   $$
	
\end{proposition}

\begin{proof}
	Let  $\mathcal{O}_{1} $ and  $\mathcal{O}_{2} $  be two non-empty open subset of  $l_{2}(\mathcal{A}) .$ As in the proof of Proposition \ref{Rni3 p2.1},  part $2) \Rightarrow 1), $ we can find some   $J,m \in \mathbb{N} ,$    $x=(x_{j})_{j} \in \mathcal{O}_{1} $  and   $ y=(y_{j})_{j} \in \mathcal{O}_{2}$  such that   $x_{j} = y_{j} =0 $   for all   $ j \neq [J]$   and   $ x_{j}=P_{m}x_{j},$   $y_{j}=P_{m}y_{j} $  for all   $j \in [J] .$ Choose the sequences $ \lbrace D_{j}^{(k)} \rbrace_{k},$ $ \lbrace G_{j}^{(k)} \rbrace_{k}$ for   $j \in [J] $  and the strictly increasing sequence   $\lbrace n_{k} \rbrace_{k} \subseteq \mathbb{N}$   that satisfy the assumptions in   $(ii) $ with respect to these   $J,m .$ For each   $k \in \mathbb{N} ,$ let   $\mu_{k}, v_{k} \in l_{2}(\mathcal{A}) $  be given by   $(\mu_{k})_{j}=D_{j}^{(k)}x_{j} ,$ $(v_{k})_{j}=G_{j}^{(k)}y_{j} $ for   $j \in [J] $   and   $(\mu_{k})_{j}=(v_{k})_{j} =0 $   for   $ j \notin [J].$ Set 
	$$\eta_{k} = \mu_{k} + T_{U,W}^{n_{k}} (v_{k}) +  S_{U,W}^{n_{k}} (v_{k}).$$
	
	By the similar calculations as in the proof of Proposition \ref{Rni3 p2.1}, part $2) \Rightarrow 1) ,$ we can show that the assumptions in $(ii)$ imply that 
	$$\lim_{k \rightarrow \infty} T_{U,W}^{n_{k}} (v_{k}) = \lim_{k \rightarrow \infty} S_{U,W}^{n_{k}} (v_{k}) =0, $$
	$$\lim_{k \rightarrow \infty} T_{U,W}^{2n_{k}} (v_{k}) = \lim_{k \rightarrow \infty} S_{U,W}^{n_{2k}} (v_{k}) =0, $$
	$$\lim_{k \rightarrow \infty} \parallel \mu_{k} - x \parallel = \lim_{k \rightarrow \infty} \parallel v_{k} - y \parallel = 0,  $$
	$$\lim_{k \rightarrow \infty} T_{U,W}^{n_{k}} (\mu_{k}) =
	\lim_{k \rightarrow \infty} S_{U,W}^{n_{k}} (\mu_{k}) = 0.$$

	It follows that  $\eta_{k} \rightarrow x $ and  $ C_{U,W}^{(n_{k})} (\eta_{k}) \rightarrow y  $ as  $k \rightarrow \infty . $ 	
\end{proof}

\begin{example} \label{Ni e3.6}
	Let $ H=L^{2}(\mathbb{R}).$ Given $m \in \mathbb{N} ,$ put for each $j,k \in \mathbb{N} $ the operator $D_{j}^{(k)} $ to be $D_{j}^{(k)}=G_{j}^{(k)} = \mathcal {L}_{\mathcal X_{[-k,k]}}P_{m}$ where $\mathcal {L}_{\mathcal X_{[-k,k]}} $ denotes the multiplication operator by $\mathcal {X}_{[-k,k]} .$ Since the convergence in the operator norm and the pointwise convergence coincide on finite dimensional spaces, it follows that $\parallel D_{j}^{(k)} - P_{m} \parallel \rightarrow 0 $ as $k \rightarrow \infty $ for all $j \in \mathbb{N} .$ If we now for each $j \in \mathbb{N} $ let $W_{j} $ be the operator from Example \ref{Rni3 e2.3}, it is not hard to see that the sufficient conditions of Proposition \ref{Rni3 p3.6} and Proposition \ref{ni pr3.5} are satisfied in this case.
\end{example}

At the end of this section we give some necessary conditions for the set of periodic elements of $T_{U,W}$ to be dense in $l_{2}(\mathcal{A}).$ For an operator $R \in B(H)$ we set 
$$m(R):= sup \lbrace C>0 \text{ } | \text{ } \parallel Rh \parallel \geq C \parallel h \parallel \text{ for all } h \in H \rbrace. $$

Further, for $J,m \in \mathbb{N}$ we let $\tilde{P}_{J,m} \in l_{2}(\mathcal{A}) $ be given as

$$(\tilde{P}_{J,m})_{j}=\left\{
\begin{array}{ll}
P_{m},& \mbox{ for } j\in[J],\\\\
0, & \mbox{ else } .
\end{array}
\right.
$$

We have the following proposition. 
\begin{proposition}
	Let $J,m \in \mathbb{N}.$ We have that (i) implies (ii).\\
	(i) $\tilde{P}_{J,m}$ belongs to the closure of $\mathcal{P}(T_{U,W}^{n})_{n}.$ \\
	(ii) There exists a strictly increasing sequence $\lbrace n_{k} \rbrace_{k} \subseteq \mathbb{N}$ such that 
	$$\lim_{k \rightarrow \infty} m ( W_{j+n_{k}} W_{j+n_{k}-1} \dots W_{j+1} )=0  $$
	for all $j \in [J].$
\end{proposition}

\begin{proof}
	Let $J,m \in \mathbb{N}$ be given. For each $k \in \mathbb{N}$ there exists by the assumption some $x^{(k)} \in l_{2} (\mathcal{A})$ and some $n_{k} \in \mathbb{N}$ such that
	$$\dfrac{1}{k^{2}} \geq \parallel x^{(k)} - \tilde{P}_{J,m} \parallel_{2} \text{ and } T_{U,W}^{n_{k}} (x^{(k)}) = x^{(k)} .$$
	Hence, for each $k \in \mathbb{N}$ and $j \in [J]$ we have that 
	$$\dfrac{1}{k^{2}} \geq \parallel x_{j}^{(k)} - {P}_{m} \parallel \geq \parallel x_{j}^{(k)} P_{m} - P_{m} \parallel, $$ 
	whic gives that 
	$$ \parallel x_{j}^{(k)} P_{m}  \parallel \geq 1 - \dfrac{1}{k^{2}} $$
	
	Thus, for each $ k \in \mathbb{N} $ and $j \in [J]$ we can find some $h_{j}^{(k)} \in H$ with $ h_{j}^{(k)} \neq 0$ such that 
	$$\parallel x_{j}^{(k)} P_{m} h_{j}^{(k)} \parallel \geq (1 - \dfrac{1}{k^{2}}) \parallel h_{j}^{(k)} \parallel. $$ 
	Now, we also have that 
	$$ \dfrac{1}{k^{2}} \geq \parallel  T_{U,W}^{n_{k}} (x^{(k)}) -  \tilde{P}_{J,m}      \parallel_{2}    $$
	since $ T_{U,W}^{n_{k}} (x^{(k)}) = x^{(k)} $ for each $k \in \mathbb{N}.$ We may in fact assume that $J < n_{1} < n_{2} < \dots$ . Hence, as  $J < n_{1} < n_{2} < \dots$ , we must have $\dfrac{1}{k^{2}} \geq \parallel (T_{U,W}^{n_{k}} (x^{(k)}))_{j} \parallel $ for all $j \in [J]$ which gives for all $j \in [J]$ and $k \in \mathbb{N}$ that
	$$\parallel W_{j+n_{k}} W_{j+n_{k}-1} \dots W_{j+1} x_{j}^{(k)} \parallel \leq \dfrac{1}{k^{2}} .  $$
	Thus,
	$$ \dfrac{1}{k^{2}} \geq \parallel W_{j+n_{k}} W_{j+n_{k}-1} \dots W_{j+1} x_{j}^{(k)} P_{m} \parallel ,$$ 
	so
	$$ \dfrac{1}{k^{2}} \parallel h_{j}^{(k)} \parallel \geq \parallel W_{j+n_{k}} W_{j+n_{k}-1} \dots W_{j+1} x_{j}^{(k)} P_{m} \parallel \text{ } \parallel h_{j}^{(k)} \parallel $$ 
	$$\geq \parallel W_{j+n_{k}} W_{j+n_{k}-1} \dots W_{j+1} x_{j}^{(k)} P_{m}  h_{j}^{(k)} \parallel   $$
	$$\geq m (  W_{j+n_{k}} W_{j+n_{k}-1} \dots W_{j+1} ) \parallel x_{j}^{(k)} P_{m}  h_{j}^{(k)}  \parallel  $$
	$$\geq (1- \dfrac{1}{k^{2}} ) \parallel h_{j}^{(k)} \parallel m ( W_{j+n_{k}} W_{j+n_{k}-1} \dots W_{j+1}  )   $$
	for all $j \in [J]$ and $k \in \mathbb{N}.$ Since $h_{j}^{(k)} \neq 0, $ we can divide on the both side of the inequality by $\parallel h_{j}^{(k)} \parallel$ and obtain that 
	$$\dfrac{1}{k^{2}-1} \geq m ( W_{j+n_{k}} W_{j+n_{k}-1} \dots W_{j+1}  )  $$
	for all $j \in [J]$ and $k \in \mathbb{N}.$
\end{proof}

Similarly we can prove the following proposition. 
\begin{proposition}
	Let $J,m \in \mathbb{N}.$ We have that (i) implies (ii).\\
	(i) $\tilde{P}_{J,m}$ belongs to the closure of $\mathcal{P}(S_{U,W}^{n})_{n}.$\\
	(ii) There exists a strictly increasing sequence $\lbrace n_{k} \rbrace_{k} \subseteq \mathbb{N}$ such that
	$$\lim_{k \rightarrow \infty} m ( W_{j-n_{k}+1}^{-1} W_{j-n_{k}+2}^{-1}  \dots W_{j}^{-1}) = 0 $$
	for all $j \in [J].$
\end{proposition}

\section{Hypercyclic operators on $C^{*}$-algebras } \label{Sec 4}
	Let $\mathcal{A}$ be a non-unital $C^{*} $-algebra such that $\mathcal{A}$ is a closed two-sided ideal in a unital $C^{*} $-algebra $\mathcal{A}_{1}.$ Let $ \Phi $ be an isometric $*$-isomorphism of $\mathcal{A}_{1}$ such that $\Phi (\mathcal{A}) = \mathcal{A} .$ Assume that there exists a net $\lbrace p_{\alpha} \rbrace_{\alpha} \subseteq \mathcal{A}  $ consisting of self-adjoint elements with $ \parallel p_{\alpha} \parallel \leq 1$ for all $\alpha $ and such that $\lbrace p_{\alpha}^{2} \rbrace_{\alpha} $ is an approximate unit for $ \mathcal{A} .$ Suppose in addition that for all $\alpha $ there exists some $N_{\alpha} \in \mathbb{N} $ such that $\Phi^{n} (p_{\alpha}) \cdot p_{\alpha} = 0 $ for all $ n \geq N_{\alpha}$ (which gives that $ 0 =( \Phi^{n} (p_{\alpha}) \cdot p_{\alpha} )^{*}= p_{\alpha} \cdot  \Phi^{n} (p_{\alpha}) $ since $ \Phi $ is a $*$-isomorphism). Let $ b \in G(\mathcal{A}_{1}) $ and $T_{\Phi , b} $ be the operator on $ \mathcal{A}_{1} $ defined by $ T_{\Phi , b}(a) = b \cdot \Phi (a) $ for all $ a \in \mathcal{A}_{1}.$ Then $T_{\Phi , b} $ is a bounded linear operator on $ \mathcal{A}_{1} $ and since $ \mathcal{A} $ is an ideal in $ \mathcal{A}_{1} ,$ it follows that $T_{\Phi , b} ( \mathcal{A}) \subseteq \mathcal{A} $ because $\Phi ( \mathcal{A} ) = \mathcal{A} .$ The inverse of $T_{\Phi , b}, $ which we will denote by $ S_{\Phi , b},$ is given as $ S_{\Phi , b} (a) = \Phi^{-1} (b^{-1}) \cdot \Phi^{-1}(a)$ for all $a \in \mathcal{A}_{1} .$ Again, since $\Phi^{-1} ( \mathcal{A} ) = \mathcal{A} $ and $ \mathcal{A} $ is a two-sided ideal in $  \mathcal{A}_{1} ,$ we have that $ S_{\Phi , b} (\mathcal{A}) \subseteq \mathcal{A} ,$ hence $ T_{\Phi , b} (\mathcal{A}) = \mathcal{A} =  S_{\Phi , b} (\mathcal{A}) .$ 

	By some calculations one can check that $  T_{\Phi , b}^{n} (a) = b \cdot \Phi(b) \dots \Phi^{n-1}(b) \Phi^{n} (a)$ and $ S_{\Phi , b}^{n} (a) =  \Phi^{-1} (b^{-1}) \Phi^{-2} (b^{-1}) \dots \Phi^{-n} (b^{-1}) \cdot \Phi^{-n} (a)  $ for all $a \in \mathcal{A} .$

\begin{proposition} \label{ni3 p31}
	The following statements are equivalent.\\
	(i) $ T_{\Phi , b} $ is hypercyclic on $ \mathcal{A} .$\\
	(ii) For every $p_{\alpha} $ there exists a strictly increasing sequence $\lbrace n_{k} \rbrace_{k} \subseteq \mathbb{N} $ and sequences $\lbrace q_{k} \rbrace_{k} ,$ $ \lbrace d_{k} \rbrace_{k} $ in $  \mathcal{A} $ such that 
	$$\lim_{k \rightarrow \infty} \parallel q_{k} - p_{\alpha}^{2} \parallel=  \parallel d_{k} - p_{\alpha}^{2} \parallel =0 $$ 
	and 
	$$ \lim_{k \rightarrow \infty} \parallel   \Phi^{-n_{k}} (b)  \Phi^{-n_{k}+1} (b) \dots  \Phi^{-1} (b) q_{k} \parallel $$ 
	$$= \lim_{k \rightarrow \infty} \parallel   \Phi^{n_{k}-1} (b^{-1})  \Phi^{n_{k}-2} (b^{-1}) \dots  \Phi (b^{-1})  b^{-1}   d_{k} \parallel = 0$$
\end{proposition}

\begin{proof}
	We prove first $ i) \Rightarrow ii) . $\\
	Let $p_{\alpha} $ be given. Since $ T_{\Phi , b} $ is hypercyclic, there exists some $ n_{1} \geq N_{\alpha} $ and some $a_{1} \in \mathcal{A} $ such that $\parallel a_{1} - p_{\alpha} \parallel < \dfrac{1}{4} $ and $ \parallel  b \cdot \Phi(b) \dots \Phi^{n_{1}-1}(b) \Phi^{n_{1}} (a_{1}) - p_{\alpha} \parallel < \dfrac{1}{4} .$ Since $ 0 = p_{\alpha}  \Phi^{n_{1}}(p_{\alpha}) = \Phi^{n_{1}} (p_{\alpha}) \cdot p_{\alpha}  ,$ we get 
	
	\begin{align*}
	\parallel \Phi^{n_{1}}(a_{1} ) \cdot p_{\alpha} \parallel  
	& = \parallel ( \Phi^{n_{1}}(a_{1} ) -  \Phi^{n_{1}}(p_{\alpha} ) ) \cdot p_{\alpha} \parallel \\
	& \leq \parallel   \Phi^{n_{1}}(a_{1} - p_{\alpha} ) \parallel\\
	& = \parallel  a_{1} - p_{\alpha} \parallel \leq \dfrac{1}{4}, \text{ so }  
	\end{align*}

	\begin{equation}\label{d1 f1}
	\parallel \Phi^{n_{1}}(a_{1} ) \cdot p_{\alpha} \parallel  \leq 
	\parallel  a_{1} - p_{\alpha} \parallel \leq \dfrac{1}{4} 
	\end{equation}
	Moreover, 
	\begin{equation}\label{d1 f2}
	\parallel  (a_{1} - p_{\alpha}) p_{\alpha}   \parallel \leq \parallel a_{1} - p_{\alpha} \parallel \leq \dfrac{1}{4}.
	\end{equation} 
	Similarly, $0=  \Phi^{-n_{1}}(p_{\alpha}) \cdot p_{\alpha} = p_{\alpha} \cdot  \Phi^{-n_{1}}(p_{\alpha}),$ so we get

	$$\parallel  \Phi^{-n_{1}}(b)  \Phi^{-n_{1}+1}(b)  \dots \Phi^{-1}(b) a_{1}p_{\alpha} \parallel$$
	$$= \parallel \Phi^{-n_{1}}(b \Phi (b)  \dots \Phi^{n_{1}-1} (b) \Phi^{n_{1}}(a_{1}) - p_{\alpha} ) p_{\alpha} \parallel $$
	$$ \leq 	\parallel  \Phi^{-n_{1}}(b \Phi (b)  \dots \Phi^{n_{1}-1} (b) \Phi^{n_{1}}(a_{1}) - p_{\alpha} ) \parallel$$
	$$= \parallel b \Phi (b)  \dots \Phi^{n_{1}-1} (b) \Phi^{n_{1}}(a_{1}) - p_{\alpha} \parallel 
	\leq \dfrac{1}{4} , \text{ so } $$

	\begin{equation}\label{d1 f3}
	\parallel  \Phi^{-n_{1}}(b)  \Phi^{-n_{1}+1}(b)  \dots \Phi^{-1}(b) a_{1}p_{\alpha} \parallel \leq \dfrac{1}{4}.
	\end{equation} 
	Finally, we have

	$$\parallel b \Phi(b) \dots \Phi^{n_{1}-1} (b) \Phi^{n_{1}}(a_{1}) p_{\alpha}   - p_{\alpha}^{2} \parallel $$
	$$= \parallel ( b \Phi(b) \dots \Phi^{n_{1}-1} (b) \Phi^{n_{1}}(a_{1}) - p_{\alpha}) \cdot p_{\alpha} \parallel $$
	$$ \leq \parallel b \Phi(b) \dots \Phi^{n_{1}-1} (b) \Phi^{n_{1}}(a_{1}) - p_{\alpha}  \parallel  \leq \dfrac{1}{4} , \text{ so } $$
	
	\begin{equation}\label{d1 f4}
	\parallel b \Phi(b) \dots \Phi^{n_{1}-1} (b) \Phi^{n_{1}}(a_{1}) p_{\alpha}   - p_{\alpha}^{2} \parallel \leq \dfrac{1}{4}.
	\end{equation}
	By (\ref{d1 f1}) we also get that 
	$$
	\parallel  \Phi^{n_{1}-1}(b^{-1})  \Phi^{n_{1}-2}(b^{-1})  \dots \Phi(b^{-1}) b^{-1}  b \Phi(b) \dots  
	\Phi^{n_{1}-1}(b)  \Phi^{n_{1}}(a_{1})p_{\alpha} \parallel $$
	$$ = \parallel \Phi^{n_{1}}(a_{1}) p_{\alpha}  \parallel
	\leq \dfrac{1}{4}
	$$ 
	Put $q_{1}=a_{1}p_{\alpha} $ and $d_{1}=b \Phi(b) \dots \Phi^{n_{1}-1}(b) \Phi^{n_{1}}(a_{1}) p_{\alpha}.$ Then $ \parallel q_{1}- p_{\alpha}^{2} \parallel < \dfrac{1}{4} ,$ 
	$ \parallel d_{1}- p_{\alpha}^{2} \parallel < \dfrac{1}{4} ,$ $\parallel  \Phi^{-n_{1}}(b)  \Phi^{-n_{1}+1}(b)  \dots \Phi^{-1}(b) q_{1} \parallel 
	\leq \dfrac{1}{4}$  and
	$$\parallel  \Phi^{n_{1}-1}(b^{-1})  \Phi^{n_{1}-2}(b^{-1})  \dots \Phi^{-1}(b^{-1})b^{-1} d_{1} \parallel \leq \dfrac{1}{4} .$$ 
	Next, since $T_{\Phi,b} $ is hypercyclic, we can find a hypercyclic vector $a_{2} $ and some $n_{2} > n_{1} $ such that $\parallel a_{2} - p_{\alpha} \parallel < \dfrac{1}{4^{2}}$ and
	$\parallel T_{\Phi,b}^{n_{2}} (a_{2}) - p_{\alpha} \parallel < \dfrac{1}{4^{2}} $ and continue as above to find $q_{2} $ and $d_{2} $ in $\mathcal{A} $ such that $\parallel q_{2} - p_{\alpha}^{2} \parallel < \dfrac{1}{4^{2}} ,$
	$\parallel d_{2} - p_{\alpha}^{2} \parallel < \dfrac{1}{4^{2}} $
	and \\
	$$\parallel  \Phi^{-n_{2}}(b)  \Phi^{-n_{2}+1}(b)  \dots \Phi^{-1}(b) q_{2} \parallel \leq \dfrac{1}{4^{2}},$$
	$$\parallel  \Phi^{n_{2}-1}(b^{-1})  \Phi^{n_{2}-2}(b^{-1})  \dots \Phi (b^{-1}) d_{2} \parallel \leq \dfrac{1}{4} .$$ 
	Proceeding inductively, we can construct the sequences $\lbrace n_{k} \rbrace_{k} , \lbrace q_{k} \rbrace_{k} $ and $  \lbrace d_{k} \rbrace_{k} $ with the properties in $ ii)$, so $i) \Rightarrow ii).$\\
	Now we prove the opposite implication.\\
	Let $\mathcal{O}_{1} $ and $ \mathcal{O}_{2}$ be two open non-empty subsets of $\mathcal{A} .$ Since $\lbrace p_{\alpha}^{2} \rbrace $ is an approximate unit in $\mathcal{A},$ we can find some $x \in \mathcal{O}_{1},$ $y \in \mathcal{O}_{2} $ such that $x=p_{\alpha}^{2}x $ and $y=p_{\alpha}^{2}y $ for so sufficiently large $\alpha .$ Choose the sequences $\lbrace n_{k} \rbrace_{k}, \lbrace q_{k} \rbrace_{k}, \lbrace d_{k} \rbrace_{k} $ satisfying the conditions of $(ii) $ with respect to $p_{\alpha}.$ For each $k \in \mathbb{N} ,$ set $x_{k}= q_{k}x + S_{\Phi,b}^{n_{k}} (d_{k}y).$ 

	We have that 
	\begin{align*}
	\parallel S_{\Phi,b}^{n_{k}} (d_{k}y) \parallel 
	&	= \parallel  \Phi^{-1}(b^{-1}) \dots  \Phi^{-n_{k}}(b^{-1})   \Phi^{-n_{k}} ( d_{k} y) \parallel    \\
	&   = \parallel \Phi^{n_{k}} (   \Phi^{-1}(b^{-1}) \dots  \Phi^{-n_{k}}(b^{-1})   \Phi^{-n_{k}} ( d_{k} y)   ) \parallel \\
	&   = \parallel  \Phi^{n_{k}-1} (b^{-1}) \dots  \Phi (b^{-1}) \cdot b^{-1} d_{k}y \parallel \\
	&   \leq \parallel  \Phi^{n_{k}-1} (b^{-1}) \dots  \Phi^{-1} (b^{-1})b^{-1} d_{k} \parallel \parallel y \parallel \rightarrow 0 \text{ as } k \rightarrow \infty. \\ 
	\end{align*}
	
	Similarly, 
	\begin{align*}
	\parallel T_{\Phi,b}^{n_{k}} (q_{k}y) \parallel 
	&	= \parallel  b \Phi(b) \dots  \Phi^{n_{k}-1}(b)   \Phi^{n_{k}} (q_{k} y) \parallel    \\
	&   = \parallel \Phi^{-n_{k}} (  b \Phi(b) \dots  \Phi^{n_{k}-1}(b)   \Phi^{n_{k}} ( q_{k} y)   ) \parallel \\
	&   \leq \parallel  \Phi^{-n_{k}} (b) \Phi^{-n_{k}+1}(b) \dots  \Phi^{-1} (b) q_{k} \parallel \parallel y \parallel \rightarrow 0 \text{ as } k \rightarrow \infty.  
	\end{align*}
	It follows that $x_{k} \rightarrow x $ and $T_{\Phi,b}^{n_{k}} (x_{k}) \rightarrow y ,$ as  $ k \rightarrow \infty , $ so $T_{\Phi,b}$ is topologically transitive, thus hypercyclic on $\mathcal{A} .$ 
\end{proof}

\begin{remark}
	We notice that the assumption that for all $\alpha$ there exists some $N_{\alpha}$ such that 
	$\Phi^{n}(p_{\alpha}) p_{\alpha} = 0$ for all $n \geq N_{\alpha}$ is in fact not needed for the proof of the implication $(ii)$ implies $(i)$ in Proposition \ref{ni3 p31}.
\end{remark}

\begin{example} \label{Ni e4.2}
	Let $H$ be a separable Hilbert space and $U$ be a unitary operator on $H$ satisfying the condition (2) from \cite{IT} with respect to an orthonormal basis $ \lbrace e_{j} \rbrace_{j \in \mathbb{Z}} .$	Set $ \Phi$ to be the $*$-isomorphism on $B(H)$ given by $\Phi (F)=U^{*}FU .$ Then, by the condition (2) from \cite{IT}, given $m \in \mathbb{N} $ there exists an $ N_{m} \in \mathbb{N}$ such that $P_{m}U^{n}P_{m}=0 $ for $n \geq N_{m} $ (where $P_{m} $ is the orthogonal projection onto $Span \lbrace e_{-m}, \dots , e_{m} \rbrace $ as in \cite{IT}.) Moreover, $\lbrace P_m \rbrace_{m \in \mathbb{N}}$ is an approximate unit for $B_{0}(H)$ by \cite[Proposition 2.2.1]{MT}. Hence, for all $ n \geq N_{m}$ we have $\Phi^{n}(P_{m})P_{m}=U^{*n}P_{m}U^{n}P_{m}=0.$ Here $\mathcal{A}_{1}=B(H) $ and $\mathcal{A}=B_{0}(H) .$ By some calculations we see that the conditions in part $(ii)$ in  Proposition \ref{ni3 p31}  are the same as the conditions (3) and (4) in \cite{IT}. The operator $T_{U,W} $ from \cite{IT} is actually the operator $T_{\Phi,WU} $ ( because $WFU=WU(U^{*}FU) $ for all $ F \in B_{0}(H) \text{ }).$ For concrete examples satisfying these conditions we refer to examples from \cite{IT} and \cite{IT2}. In fact, in \cite{IT2} it has been proved that these conditions are equivalent to the condition that the operator $W$ satisfies hypercyclicity criterion on $H.$ For more details about this criterion, see \cite{bp07}.

\end{example}

\begin{example}
	Let $H=L^{2}(\mathbb{R}) .$  For each  $j,k,m \in \mathbb{N} $ we let  $D_{j}^{(k)},P_{m} $  be the operators on  $H $ as in Example \ref{Ni e3.6} and for each  $J \in \mathbb{N} $ we let  $ \tilde{P}_{J,m}$  be the orthogonal projection on  $ l_{2}(B_{0}(H))$  induced by  $P_{m} $ and  $[J] ,$ as defined on page 10 in Section \ref{Section 3}. Let  $\mathcal{K} (l_{2}(B_{0}(H))) $ denote the  $C^{*}$-algebra of compact operators on  $l_{2}(B_{0}(H)) $  in the sense of \cite[Section 2.2]{MT}. Then it is not hard to see that  $\lbrace \tilde{P}_{J,m} \rbrace_{J.m \in \mathbb{N}} $  is an approximate unit for  $\mathcal{K} (l_{2}(B_{0}(H))) .$ For  $j \in \mathbb{Z} $  we let  $W_{j} $  be the operator on  $H $ from Example \ref{Rni3 e2.3}. Let  $T_{U,W} $ be the operator on  $l_{2}(B_{0}(H)) $  defined in Section \ref{Section 3}, where  $W=\lbrace W_{j} \rbrace_{j \in \mathbb{Z}}.$ If  $U=I,$ then  $T_{I,W} $  is a bounded, adjointable operator on  $l_{2}(B_{0}(H)) $ which is linear with respect to the  $C^{*} $-algebra  $B_{0}(H) .$ (Recall that we consider  $l_{2}(B_{0}(H)) $ as the right Hilbert  $C^{*} $-module). For each  $k \in \mathbb{N} ,$ set  $\tilde{D}_{k} $ to be the operator on  $l_{2}(B_{0}(H))  $ given by  $\tilde{D}_{k} (\lbrace x_{j} \rbrace_{j \in \mathbb{Z}}) = \lbrace D_{j}^{(k)} x_{j} \rbrace_{j \in \mathbb{Z}} $ for all  $\lbrace  x_{j} \rbrace_{j \in \mathbb{Z}} \in l_{2}(B_{0}(H))  $. Since  $\parallel D_{j}^{(k)} \parallel \leq 1  $ for all  $j \in \mathbb{Z} $  and  $ k \in \mathbb{N},$ we have that  $\tilde{D}_{k} $ is a bounded  $B_{0}(H) $-linear, adjointable operator on $l_{2}(B_{0}(H))$  for all  $ k \in \mathbb{N}.$ By the similar arguments as in Example \ref{Ni e3.6} we can deduce that  
	$$\lim_{k \rightarrow \infty}\parallel \tilde{D}_k \tilde{P}_{J,m} -\tilde{P}_{J,m} \parallel =0$$
	for all  $J,m \in \mathbb{N} .$ Moreover, for all  $ k,J,m \in \mathbb{N} $ we have that  
	$$\lim_{n \rightarrow \infty}\parallel T_{I,W}^{n} \tilde{D}_{k}P_{J,m} \parallel =
	\lim_{n \rightarrow \infty}\parallel T_{I,W}^{-n} \tilde{D}_{k} \tilde{P}_{J,m} \parallel = 0 $$ 
	
	Hence, for all  $J,m \in \mathbb{N} $ we can construct a strictly increasing sequence  $ \lbrace n_{k} \rbrace_{k} \subseteq \mathbb{N}$  such that  
	$$0=\lim_{k \rightarrow \infty}\parallel T_{I,W}^{n_{k}} \tilde{D}_{k} \tilde{P}_{J,m} \parallel = 
	\lim_{k \rightarrow \infty}\parallel T_{I,W}^{-n_{k}} \tilde{D}_{k} \tilde{P}_{J,m} \parallel = 0   $$ 
	Let now $ \mathcal{A} = \mathcal{K}(l_{2}(B_{0}(H))) $ and  $\mathcal{A}_{1} $ be the  $C^{*}$-algebra of all bounded  $B_{0}(H) $-linear, adjointable operators on  
	$l_{2}(B_{0}(H)) .$ If  $\tilde{U} $ is a unitary operator on  $l_{2}(B_{0}(H)) ,$ we let   $\Phi $ be the  $* $-isomorphism on  $\mathcal{A}_{1} $  given by  $\Phi (F) = \tilde{U}^{*}F\tilde{U}$  for all  $F \in \mathcal{A}_{1} .$ Put then  $ b= T_{I,W}\tilde{U} \in  G(\mathcal{A}_{1}).$ By the same arguments as in Example \ref{Ni e4.2} we can deduce that the conditions of Proposition \ref{ni3 p31} are satisfied in this case.         
\end{example}

\begin{example}
	Let $X$ be a locally compact Hausdorff space,  $\mathcal{A} = C_{0}(X) ,$ $\mathcal{A}_{1} = C_{b}(X) $ and $\Phi$ be given by $\Phi (f) = f \circ \alpha $ for all $ f \in  C_{b}(X) $ where $\alpha $ is a homeomorphism of $X.$ Put 
	$$ S= \lbrace f \in C_{c}(X) \text{ } \vert \text{ } 0 \leq f \leq 1 \text{ and }  f_{\mid_{K}}=1 \text{ for some compact } K \subset X \rbrace .$$ 
	If $\tilde{S} = \lbrace f^{2} \text{ } \vert \text{ } f \in S  \rbrace ,$ then $\tilde{S}$  is an approximate unit for $ C_{0}(X) .$ Suppose that $\alpha$ is \emph{aperiodic}, that is for each compact subset $K$ of $X$, there exists a constant $N>0$ such that for each $n\geq N$, we have $K \cap \alpha^{n}(K)=\varnothing$. By some calculations it is not hard to see that in this case the conditions in Proposition \ref{ni3 p31} are equivalent to the condition that for every compact subset $K$ of $\Omega$ there exists a strictly increasing sequence  $\lbrace n_{k} \rbrace_{k} \subseteq \mathbb{N} ,$ such that 
	$$0=\lim_{k \rightarrow \infty} ( \sup_{t \in K}  \prod_{j=0}^{n_{k}-1} (b \circ \alpha^{j-n_{k}}) (t)  ) = \lim_{k \rightarrow \infty} ( \sup_{t \in K}  \prod_{j=0}^{n_{k}-1} (b \circ \alpha^{j})^{-1} (t)  ) ,$$ 
	For the concrete examples satisfying these conditions, we refer to examples in \cite{si09}.
\end{example}

If $a \in \mathcal{A}_{1},$ in the sequel we shall denote by $L_{a}$ the left multiplier by $a.$

\begin{corollary}
	If there exist dense subsets $\Omega_{1} $ and $\Omega_{2} $ of $\mathcal{A} $ and a strictly increasing sequence $\lbrace n_{k} \rbrace_{k} \subseteq \mathbb{N}  $ such that 
	$$ L_{\Phi^{-n_{k}} (b) \Phi^{-n_{k}+1}(b) \dots  \Phi^{-1} (b)} \stackrel{ k \rightarrow \infty }{\longrightarrow} 0   $$ 
	pointwise on $\Omega_{1} $ and 
	$$ L_{\Phi^{n_{k}-1} (b^{-1}) \Phi^{n_{k}-2} (b^{-1}) \dots  \Phi (b^{-1})  b^{-1} } \stackrel{ k \rightarrow \infty }{\longrightarrow} 0   $$ 
	pointwise on $\Omega_{2} ,$ then $T_{\Phi,b} $ is hypercyclic on $\mathcal{A} .$ 
\end{corollary}

\begin{proof}
	Let $p_{\alpha} $ be given. Since $\Omega_{1} $ and $\Omega_{2} $ are dense in $\mathcal{A} ,$ there exist some $q_{1} \in \Omega_{1} $ and $d_{1} \in \Omega_{2}  $ such that 
	$$\parallel q_{1} - p_{\alpha}^{2} \parallel < \dfrac{1}{4} \text{ and } \parallel d_{1} - p_{\alpha}^{2} \parallel < \dfrac{1}{4}  .$$  
	By the assumption we can find some $n_{k_{1}} $ such that 
	$$\parallel \Phi^{-n_{k}} (b) \Phi^{-n_{k}+1}(b) \dots  \Phi^{-1} (b) q_{1} \parallel < \dfrac{1}{4}  $$ 
	and 
	$$\parallel \Phi^{n_{k}-1} (b^{-1}) \Phi^{n_{k}-2} (b^{-1}) \dots  \Phi (b^{-1})  b^{-1}  d_{1} \parallel < \dfrac{1}{4} $$ 
	for all $k \geq k_{1} .$ Then we find some $q_{2} \in \Omega_{1}, d_{2} \in \Omega_{2}  $ such that 
	$$\parallel q_{2} - p_{\alpha}^{2} \parallel < \dfrac{1}{4^{2}} \text{ and } \parallel d_{2} - p_{\alpha}^{2} \parallel < \dfrac{1}{4^{2}}  .$$   
	By the assumption, we can find some $k_{2} \geq k_{1} $ such that
	$$\parallel \Phi^{-n_{k}} (b) \Phi^{-n_{k}+1}(b) \dots  \Phi^{-1} (b) q_{2} \parallel < \dfrac{1}{4^{2}}  $$ 
	and 
	$$\parallel \Phi^{n_{k}-1} (b^{-1}) \Phi^{n_{k}-2} (b^{-1}) \dots  \Phi (b^{-1})  b^{-1}  d_{2} \parallel < \dfrac{1}{4^{2}} ,$$ 
	for all $k \geq k_{2} .$ Proceeding inductively, we can construct the strictly increasing sequence $ \lbrace n_{k_{i}} \rbrace_{i} $ and the sequences  $ \lbrace q_{i} \rbrace_{i} $  in  $ \lbrace d_{i} \rbrace_{i} $ in $\mathcal{A} $  satisfying the conditions of Proposition \ref{ni3 p31}.  
\end{proof}

\textbf{Acknowledgement}

I am grateful to Professor Dragan S.  \DJ{}or\dj{}evi\'{c} for suggesting to me Hilbert $C^{*}$-modules as the main field of my research and for introducing to me the relevant literature. Also, I am grateful to Professor S.M. Tabatabaie for suggesting to me linear dynamics of operators as the topic of my research and for introducing to me the relevant literature.

\vspace{.1in}
\end{document}